\documentclass[12pt,centertags,oneside,reqno]{amsart}
\RequirePackage[utf8]{inputenc}
\usepackage{amstext,amsthm,amscd,typearea,hyperref}
\usepackage{amsmath}
\usepackage{amssymb}
\usepackage{a4wide}
\usepackage[mathscr]{eucal}
\usepackage{mathrsfs}
\usepackage{typearea}
\usepackage{charter}
\usepackage{pdfsync}
\usepackage[a4paper,width=16.2cm,top=3cm,bottom=3cm]{geometry}

\usepackage{multicol}

\numberwithin{equation}{section}

\emergencystretch=\maxdimen
\newtheorem{theorem}{Theorem}[section]

\newtheorem{proposition}[theorem]{Proposition}
\newtheorem{corollary}[theorem]{Corollary}
\newtheorem{lemma}[theorem]{Lemma}
\newtheorem{remark}[theorem]{Remark}

\newcommand{\cali}[1]{\mathscr{#1}}

\newcommand{\dist}{\mathop{\mathrm{dist}}\nolimits}

\newcommand{\ddc}{dd^c}

\newcommand{\dbar}{\overline\partial}

\newcommand{\ep}{\epsilon}

\newcommand{\Cc}{\cali{C}}
\newcommand{\Dc}{\cali{D}}

\newcommand{\Fc}{\cali{F}}

\newcommand{\Uc}{\cali{U}}

\newcommand{\C}{\mathbb{C}}

\newcommand{\R}{\mathbb{R}}

\renewcommand\P{\mathbb{P}}

\newcommand{\lp}{\langle}
\newcommand{\rp}{\rangle}

\newcommand{\dsh}{\mathrm{DSH}}

\title{Exponential mixing property for automorphisms of compact K\"ahler manifolds}

\author{Hao Wu}
\address{Department of Mathematics,  National University of Singapore - 10, Lower Kent Ridge Road - Singapore 119076}
\email{e0011551@u.nus.edu}

\date{}

\begin{document}
	
	\begin{abstract}
Let $f$ be a holomorphic automorphism of a compact K\"ahler manifold. Assume that $f$ admits a unique maximal dynamic degree $d_p$ with only one eigenvalue of maximal modulus. Let $\mu$ be its equilibrium measure.  In this paper, we prove that $\mu$ is exponentially mixing for all d.s.h.\ test functions.
   \end{abstract}
	
	\maketitle
	
	\medskip
	
	\noindent {\bf Classification AMS 2010}: 37F, 32H.
	
	\medskip
	
	
	\medskip

\section{Introduction and main results}

Let $(X,\omega)$ be a compact K\"ahler manifold of dimension $k$ and let $f$ be a holomorphic automorphism of $X$. Denote by $f^*$ the pull-back operator acting on the Hodge cohomology groups $H^{*,*}(X,\C)$. Recall that the \textit{dynamic degree of order $q$} of $f$ is the spectral radius of $f^*$ on   $H^{q,q}(X,\C)$, and denoted by $d_q$. We have $d_0=d_k=1$. Khovanskii-Teissier-Gromov \cite{gromov2} proved that the function $q\mapsto \log d_q$ is concave. Thus there are integers $0\leq p\leq p'\leq k$ such that $$1=d_0<\cdots< d_p=\cdots=d_{p'}>\cdots>d_k=1.$$ 

When $p=p'$ and in addition, when $f^*$ acting on $H^{p,p}(X,\C)$, admits only one eigenvalue of maximal modulus (necessary equal to $d_p$), there is a  unique invariant probability measure $\mu:=T_+\wedge T_-$, where $T_+$  is the \textit{Green $(p,p)$-current} of $f$ and $T_-$ is the Green $(k-p,k-p)$-current of $f^{-1}$. They satisfy $f^*(T_+)=d_pT^+$ and $f_*(T_-)=d_{k-p}T_-$. Moreover, for any positive closed $(p,p)$-current (resp. $(k-p,k-p)$-current) $S$ of mass $1$, we have $d_p^{-n}(f^n)^*(S)$ (resp. $d_{k-p}^{-n}(f^n)_*(S)$) converge to $T_+$ (resp. $T_-$). And $T_+$ (resp. $T_-$) is the unique positive closed current in the class $\{T_+\}$ (resp. $\{T_-\}$).  The measure $\mu$ is called the \textit{equilibrium measure} of $f$.  For the constructions of $\mu,T_+,T_-$, the readers may refer to \cite{super-kahler}. And see e.g.\ \cite{oguiso-entropy,oguiso-tr} for interesting examples.

Recall that a function  is \textit{quasi-plurisubharmonic} (\textit{quasi-p.s.h.}\ for short) on $X$ if locally it is the difference of a plurisubharmonic (p.s.h.\ for short) function and a smooth one.  The following theorem is our first main result.

\begin{theorem}\label{main}
	Let $f$ be a holomorphic automorphism on a compact K\"ahler manifold $X$ of dimension $k$ and let $\mu$ be its equilibrium measure. Let $d_q$ be the dynamic degree of order $q$, $0\leq q\leq k$. Assume that there is a integer $p$ such that $d_p$ is strictly large than other dynamic degrees and $d_p$ admits only one eigenvalue of maximal modulus $d_p$. Then $\mu$ is exponentially mixing for bounded quasi-p.s.h.\ observables. More precisely, if $\delta$ is a constant such that $\max\{d_{p-1},d_{p+1}\}<\delta<d_p$ and all the eigenvalues of $f^*$, acting on $H^{p,p}(X,\C)$, except $d_p$, are strictly smaller than $\delta$. Then there exists a constant $c>0$,   such that 
	$$\Big| \int (\varphi\circ f^n)\psi d\mu-\Big(\int \varphi d\mu\Big)\Big(\int \psi d\mu\Big)\Big|\leq c(d_p/\delta)^{-n/2}\|\varphi\|_{L^\infty}\|\psi\|_{L^\infty}$$ for all $n\geq 0$ and all  bounded quasi-p.s.h.\ functions $\varphi$ and $\psi$ satisfy $\ddc\varphi\geq -\omega,\ddc\psi\geq -\omega$.
\end{theorem}

The conditions $\ddc\varphi\geq -\omega,\ddc\psi\geq -\omega$ in Theorem \ref{main} relate to the $*$-norm defined in section \ref{section2}.  Another version of Theorem \ref{main} has been proved in \cite{exp-kahler} for $\varphi,\psi\in \Cc^2$ and it can be extended to $\Cc^\alpha$ case, $0<\alpha\leq 2$, using interpolation theory between Banach spaces. In this case, one considers the new system $(z,w)\mapsto \big(f^{-1}(z),f(w)\big)$ on $X\times X$ and the test function $\varphi(z)\psi(w)$, which plays a linear ``role" in the new system. Since $\varphi(z)\psi(w)$ is of class $\Cc^2$ and in particular, it is H\"older continuous, some estimates of super-potentials on the currents with H\"older continuous super-potentials  imply the desired result.

However, in the study of complex dynamics, sometimes we need to investigate the behaviors of the functions with some singularities under the action of $f$. For example, the class of quasi-p.s.h.\ functions or d.s.h.\ functions (see the definition below). When $\varphi$ and $\psi$ are not of class of $\Cc^2$, then  idea in \cite{exp-kahler} can not be directly applied. In this case, firstly, the super-potentials may not be well defined on the space of non-smooth currents. Secondly, when $\varphi$ and $\psi$ are not in $\Cc^2$,   the function $\varphi(z)\psi(w)$ will not be H\"older continuous any more. In the proof, we  do some regularization of quasi-p.s.h.\ functions. After that we combine the idea in \cite{exp-kahler} with some  techniques in \cite{wu-mix} to prove the main theorem.  Similarly  estimates of super-potentials on the currents with H\"older continuous super-potentials also are obtained at the end of section \ref{section2}. \\

Recall that a function  $u$ on $X$ with values in $\R\cup\{\pm\infty\}$ is said to be \textit{d.s.h.}\ if outside a pluripolar set it is equal to a difference of two quasi-p.s.h.\ functions. Two d.s.h.\ functions are identified when they are equal out of a pluripolar set.   Denote the set of d.s.h.\ functions by $\dsh(X)$.  Clearly it is a vector space and equips with a norm \[\|u\|_{\dsh}:=\Big|\int_{X} u\omega^k\Big|+\min\|T^{\pm}\|,\] where  the minimum is taken on all  positive closed $(1,1)$-currents $T^{\pm}$ such that  $\ddc u=T^+-T^-$.

A positive measure $\nu$ on $X$ is said to be \textit{moderate} if  for any bounded family $\Fc$ of d.s.h.\ functions on $X$, there are constants $\alpha>0$ and $c>0$ such that \[\nu\{z\in X:|\psi(z)|>M\}\leq ce^{-\alpha M}\] for $M\geq 0$ and $\psi\in\Fc$ (see \cite{dinh-exponential,dinh-dynamique,dinh-sibony:cime}). The papers \cite{dinh-exponential,super-kahler} show that if $f$ is a holomorphic automorphism of a compact K\"ahler surface or more generally, on a compact K\"ahler manifold, then the equilibrium measure $\mu$ of $f$ is moderate. Using the moderate property of $\mu$ and following the same approach as in the proof of   \cite[Theorem 1.3]{wu-mix}, we get the following theorem, which removes the boundedness conditions of $\varphi$ and $\psi$.

\begin{theorem}\label{second-theorem}
Let $f,d_p,\mu$ be as in Theorem \ref{main}. Then the equilibrium measure $\mu$ is exponentially mixing for all  d.s.h\ observables. More precisely, if $\delta$ is a constant such that $\max\{d_{p-1},d_{p+1}\}<\delta<d_p$ and all the eigenvalues of $f^*$, acting on $H^{p,p}(X,\C)$, except $d_p$, are strictly smaller than $\delta$. Then for any two d.s.h.\ functions $\varphi,\psi$, there exists a constant $c>0$,   such that	$$\Big| \int (\varphi\circ f^n)\psi d\mu-\Big(\int \varphi d\mu\Big)\Big(\int \psi d\mu\Big)\Big|\leq c(d_p/\delta)^{-n/2}$$ for all $n\geq 0$.
\end{theorem}

In Theorem \ref{second-theorem}, the constant $c$ depends on $\varphi$ and $\psi$. It is not hard to see that we can take a common $c$ for any compact family of d.s.h.\ functions.\\

Now we consider a particular case. When $X$ is a compact K\"ahler surface and  $f$ is of positive entropy, Gromov \cite{gromov} and Yomdin \cite{yomdin} showed that the entropy  is equal to $\log d_1$. Thus in this case, $d_1>1$. Moreover, Cantat \cite{cantat} proved that  the eigenvalues of $f^*$, acting on $H^{1,1}(X,\C)$, are   $d_1,1/d_1$ and others with modulus $1$. Thus we get the following corollary.

\begin{corollary}
	Let $f$ be a holomorphic automorphism of positive entropy on a compact K\"ahler surface $X$. Then the equilibrium measure $\mu$ is exponentially mixing for all d.s.h.\ observables.
\end{corollary}

In this paper, the symbols $\lesssim$ and $\gtrsim$ stand for inequalities up to a multiplicative constant. 

	\noindent\textbf{Acknowledgements:} This work was supported by the grant:  AcRF Tier 1 R-146-000-248-114 from National University of Singapore.  I also would like to thank the referees
	for their remarks. \\

\vspace{0.5cm}

\section{Super-potentials of currents}\label{section2}
In this section, we will introduce the notion called super-potential. The readers may refer to \cite{dinh-sibony:acta,super-kahler} for details. Some estimates of super-potentials on a family of  currents with H\"older continuous super-potentials are obtained at the end of this section.

Denote by $\Dc_q$ the real space that generated by all positive closed  $(q,q)$-currents on $X$. Define a norm $\|\cdot\|_*$ on $\Dc_q$ by $$\|\Omega\|_*:=\min\{\|\Omega^+\|+\|\Omega^-\|\},$$ where $\|\Omega^\pm\|:=\lp\Omega^\pm,\omega^{k-q}\rp$ are the mass of $\Omega^\pm$, and the minimum is taken over all the positive closed currents $\Omega^\pm$ with $\Omega=\Omega^+-\Omega^-$. Observe that $\|\Omega^\pm\|$ only depend on the cohomology classes of $\Omega^\pm$ in $H^{q,q}(X,\R)$. We have the following lemma.

\begin{lemma}\label{2.1}
		Let $\Omega$ be a real $\ddc$-exact $(q,q)$-current on $X$ and assume $\Omega\geq-S$ for some positive closed $(q,q)$-current $S$, then $\|\Omega\|_*\leq 2\|S\|$.
\end{lemma}

\begin{proof}
	Note that $\Omega+S$ is a positive closed current   and we can write $\Omega$ as \[ \Omega= (\Omega+S)-S.\] The mass of $\Omega+S$ is $\|S\|$ because $\Omega$ is $\ddc$-exact.
\end{proof}

We introduce the \textit{$*$-topology} on $\Dc_q$: for  a sequence of currents $S_n$ in $\Dc_q$, we say $S_n$ converge to a current $S$ in $\Dc_q$ if $S_n$ converge to $S$ in the sense of currents and if $\|S_n\|_*$ are uniformly bounded. Note that smooth forms are dense in $\Dc_q$ for this topology.

Let $\Dc_q^0$ be the subspace of $\Dc_q$ which contains all the currents of class $\{0\}$ in $H^{q,q}(X,\R)$. It is not hard to see $\Dc_q^0$ is closed under the above topology.\\

Now we define the super-potential of a current $S\in \Dc_q$. Fix a basis of $H^{q,q}(X,\R)$, denoted by $\{\alpha\}:=\big\{\{\alpha_1\},\dots,\{\alpha_t\}\big\}$. We can take all the $\alpha_j$ being smooth forms. For any $R\in \Dc_{k-q+1}^0$, there exists a real $(k-q,k-q)$-current $U_R$ such that $\ddc U_R=R$. We call $U_R$ a \textit{potential} of $R$. After adding some closed form to $U_R$ we can assume $\lp U_R,\alpha_j\rp=0$ for all $1\leq j\leq t$. After that we say  $U_R$ is \textit{$\alpha$-normalized}. If in addition,  $R$ is smooth, then we can choose $U_R$ smooth.

The \textit{$\alpha$-normalized super-potential} $\Uc_S$ of $S$ is a linear functional on the smooth forms in $\Dc_{k-q+1}^0$, and it is  defined by  $$\Uc_S(R):=\lp S,U_R\rp,$$ where $U_R$ is a smooth $\alpha$-normalized potential of $R$. Note that $\Uc_S(R)$ does not depend on the choice of $U_R$.

If $\Uc_S$ can be extended continuously to a linear functional on $\Dc_{k-q+1}^0$ for the $*$-topology we defined above, then we say $S$ has a \textit{continuous super-potential}. If $S\in \Dc_q^0$, then $\Uc_S$ does not depend on the choice of $\alpha$. If $S$ is smooth, then it has a continuous super-potential and we have  $\Uc_S(R)=\Uc_R(S)$, where $\Uc_R$ is the  super-potential of $R$. The equality still holds if we only assume $S$ has a continuous super-potential (see \cite{super-kahler}).\\

For $0<l<\infty$, we define the norm $\|\cdot\|_{\Cc^{-l}}$ and the distance $\dist_l$ on $\Dc_q$ by $$\|\Omega\|_{\Cc^{-l}}:=\sup_{\|\Phi\|_{\Cc^l}\leq 1} |\lp\Omega,\Phi\rp|\quad\text{and}\quad \dist_l(\Omega,\Omega'):=\|\Omega-\Omega'\|_{\Cc^{-l}},$$ where $\Phi$ is a smooth test $(k-q,k-q)$-form  on $X$. For $0<l<l'<\infty$, on any $\|\cdot\|_*$-bounded subset of $\Dc_p$, we have $$\dist_{l'}\leq \dist_l\leq c_{l,l'}(\dist)^{l/l'}$$ for some positive constant $c_{l,l'}$ (see\cite{super-kahler}).

For $S\in \Dc_q$ and  constants $l>0,0<\lambda\leq 1,M\geq 0$, a  super-potential $\Uc_S$ of $S$ is said to be \textit{$(l,\lambda,M)$-H\"older continuous} if it is continuous and it satisfies $$|\Uc_S(R)|\leq M\|R\|_{\Cc^{-l}}^\lambda$$ for all $R\in \Dc_{k-q+1}^0$ with $\|R\|_*\leq 1$. If $l'>0$ is another constant, the above identity for $\dist_l$ and $\dist_{l'}$ implies that $\Uc_S$ is also $(l',\lambda',M')$-H\"older continuous for some constants $\lambda'$ and $M'$ independent of $S$. And this definition does not depend on the normalization of the super-potential. We need the following two lemmas  which are originally stated in \cite{exp-kahler}.

\begin{lemma}\label{lem1}
Let $R\in\Dc_{k-p+1}^0$ with  $\|R\|_*\leq 1$ and $\Uc_R$ is $(2,\lambda,M)$-H\"older continuous. There is a constant $A>0$ independent of $R,\lambda$ and $M$ such that the  super-potential $\Uc_S$ of $S$ satisfies $$|\Uc_S(R)|\leq A(1+\lambda^{-1}\log^+M),$$ for any $S\in\Dc_p^0$ with $\|S\|_*\leq 1$, where $\log^+:=\max\{0,\log\}$.
\end{lemma}

\begin{lemma}\label{lem2}
	Let $f,p$ be as in Theorem \ref{main} and let $R\in\Dc_{k-p+1}^0$ whose super-potential $\Uc_R$ is $(2,\lambda,M)$-H\"older continuous. Then there is a constant $A_0\geq 1$ independent of $R,\lambda,M$ such that the super-potential $\Uc_{f_*(R)}$ of $f_*(R)$ is $(2,\lambda,A_0M)$-H\"older continuous.
\end{lemma}

We will use the above two lemmas to show the following result. A simple case was shown in   \cite[Proposition 3.1]{exp-kahler}, which is crucial in the proof of exponential mixing theorem for $\Cc^{\alpha}$ observables for $0<\alpha\leq 2$.  Since $T_+$ is the unique positive current in $\{T_+\}$, if $S\in \Dc_p$, then $d_p^{-n}(f^n)^*(S)$ converge to a multiple of $T_+$. 

\begin{proposition}\label{mainlemma}
	Let $f,d_p,\delta$ be as in Theorem \ref{main} and $S\in \Dc_p$. Let $r$ be the constant such that $d_p(f^n)^*(S)$ converge to $rT_+$. Let $\{R_{\ep}\}_{0<\ep\leq 1/2}$ be a family of currents in $\Dc_{k-p+1}^0$ with $\|R_\ep\|_*\leq 1$ whose super-potentials $\Uc_{R_\ep}$ are $(2,\lambda,\ep^{-2})$-H\"older continuous. Let $\Uc_n$ and $\Uc_+$ be the $\alpha$-normalized super-potential of  $d_p^{-n}(f^n)^*(S)$ and $T_+$ respectively. Then there exists a constant $A>0$ independent of the family $\{R_{\ep}\}$  such that $$|\Uc_n(R_\ep)-r\Uc_+(R_\ep)|\leq -A\log\ep(d_p/\delta)^{-n}$$ for all $n$ and $\ep$.
\end{proposition}

\begin{proof}
	It was shown in \cite[Section 3]{exp-kahler} and \cite[Section 4]{super-kahler} that for $S\in \Dc_p$ smooth and closed, we have $|\Uc_n(R)-r\Uc_n(R)|\lesssim (d_p/\delta)^{-n}\|R\|_*$ for all $R\in \Dc_{k-p+1}^0$. So we can subtract a smooth closed $(p,p)$-form from $S$ and assume that $S\in \Dc_p^0$ and $r=0$.

	Fix a constant $\delta_0$ such that $\max\{d_{p-1},d_{p+1}\}<\delta_0<\delta$ and $\delta_0$ satisfies the same properties of $\delta$ as in Theorem \ref{main}. By Poincar\'e duality, the dynamic degree $d_{p-1}$ of $f$ is equal to the dynamic degree $d_{k-p+1}(f^{-1})$ of $f^{-1}$. Since the mass of a positive current can be computed cohomologically, we have $\|(f^n)_*(R_\ep)\|_*\lesssim \delta_0^n\|R_\ep\|_*$. 
	
	Define $R_{n,\ep}:=c^{-1}\delta_0^{-n}(f^n)_*(R_\ep)$ where $c\geq 1$ is a fixed constant large enough such that $\|R_{n,\ep}\|_*\leq 1$ for all $n$ and $\ep$. By Lemma \ref{lem2}, the super-potential of $R_{n,\ep}$, denoted by $\Uc_{R_{n,\ep}}$, is $(2,\lambda, A_0^n\ep^{-2})$-H\"older continuous for some $A_0\geq 1$. On the other hand, since $S\in \Dc_p^0$, by  definition  we have
	$$\Uc_n(R_\ep)=d_p^{-n}\Uc_S\big((f^n)_*(R_\ep)\big)=c(d_p/\delta_0)^{-n}\Uc_S(R_{n,\ep}).$$ 
	Finally, applying Lemma \ref{lem1}, we obtain
	$$|\Uc_n(R_\ep)|=c(d_p/\delta_0)^{-n}|\Uc_S(R_{n,\ep})|\lesssim(d_p/\delta_0)^{-n}\big(1+\lambda^{-1}\log^+(A_0^n\ep^{-2})\big)\lesssim -\log\ep(d_p/\delta)^{-n}.$$ This finishes the proof.
\end{proof}

\vspace{0.5cm}

\section{Exponentially mixing of $\mu$}

From now on, let $f,d_p$ and $\delta$ be as in Theorem \ref{main}, and let $S$ be a fixed  positive closed $(p,p)$-current of mass $1$ on $X$. Define a sequence of currents $S_n$ by $S_n:=d_p^{-n}(f^n)^*(S)$.  We know that $S_n$ converge to $T_+$. Fix a basis $\{\alpha\}:=\big\{\{\alpha_1\},\dots,\{\alpha_t\}\big\}$ of $H^{p,p}(X,\R)$. Denote by $\Uc_n$ and $\Uc_+$    be the $\alpha$-normalized super-potentials of  $S_n$ and $T_+$  respectively.

For any bounded quasi-p.s.h.\ function $\phi$ on $X$ such that $\ddc \phi\geq -\omega,|\phi|\leq 1$.  We consider the same  regularization of $\phi$ as in \cite[Theorem 2.1]{dinh-fekete} (when $X=\P^k$, see also \cite[Section 2]{dinh-sibony:acta}), which is using a standard convolution and a partition of
unity to regularize the function locally, then gluing them globally by using maximal regularization function \cite[I.5]{demailly:agbook}. So there exists a family of smooth functions $\phi_\ep,0<\ep\leq 1/2$ such that $\ddc\phi_\ep\geq -\omega$, and  $\phi_\ep$ decreases to $\phi_0:=\phi$ when $\ep$ decreases to $0$. And $\phi_\ep$ satisfies the following two estimates:
\begin{equation} \label{regul}
\|\phi_\ep-\phi\|_{L^1(\omega^k)}\lesssim \ep \quad\text{and}\quad \|\phi_\ep\|_{\Cc^2}\lesssim \ep^{-2},
\end{equation} where the $\lesssim$'s are independent of $\phi$.

We define a sequence of functions $h_n$ and $h$ on $(0,1/2]$ by 
$$h_n(\ep)=\Uc_n(\ddc\phi_\ep\wedge T_-)\quad\text{and}\quad h(\ep)=\Uc_+(\ddc\phi_\ep \wedge T_-).$$
By definition, $$h_n(\ep)=\lp S_n\wedge T_-,\phi_\ep\rp-\lp S_n, K_\ep\rp\quad \text{and} \quad h(\ep)=\lp T_+\wedge T_-,\phi_\ep\rp-\lp T_+, K_\ep\rp,$$ where $K_\ep$ is a smooth closed $(k-p,k-p)$-form depends on $\ep$ such that $\phi_\ep T_--K_\ep$ is the $\alpha$-normalized potential of $\ddc\phi_\ep\wedge T_-$, i.e.\  $\lp\phi_\ep T_--K_\ep,\alpha_j\rp=0$ for all $j$. Observe that  $h_n$ converge pointwise to $h$ on $(0,1/2]$. 

On the other hand, note that   $\{\omega^k\}$ is a basis of $H^{k,k}(X,\R)$. We consider the $\{\omega^k\}$-normalized super potential of $\mu=T_+\wedge T_-$ and define the  function $$g(\ep):=\Uc_\mu(\ddc \phi_\ep)=\lp T_+\wedge T_-, \phi_\ep\rp -\lp  \omega^k,\phi_\ep\rp.$$ 
The function $g$ is well defined at $\ep=0$ because $T_+\wedge T_-$ has   a H\"older continuous super-potential (see \cite{super-kahler}). We prove two  lemmas first.

\begin{lemma}\label{3.0}
	There exists a constant $c>0$ independent of $\phi$ such that $$|\lp S_n, K_\ep\rp-\lp T_+, K_\ep\rp|\leq c (d_p/\delta)^{-n}$$ for $\ep\in (0,1/2]$.
\end{lemma}

\begin{proof}
	Let $(a_{n,1},a_{n,2},\dots,a_{n,t})$ be the vector which represents the class $\{S_n\}$ in $H^{p,p}(X,\R)$ with respect to the basis $\{\alpha\}$, i.e. $\{S_n\}=\sum_{j=1}^ta_{n,j}\{\alpha_j\}$. Let $(b_1,b_2,\dots,b_t)$ be the vector which represents $\{T_+\}$. Since $K_\ep$ is closed, we have $$\lp S_n-T_+, K_\ep\rp= \sum_{j=1}^t\big\lp (a_{n,j}-b_j)\alpha_j, K_\ep\big\rp.$$  Combining with $\lp\phi_\ep T_--K_\ep,\alpha_j\rp=0$ for all $j$, we get $$\lp S_n-T_+, K_\ep\rp=\sum_{j=1}^t (a_{n,j}-b_j)\lp\alpha_j, \phi_\ep T_-\rp.$$
	
	On the other hand, by Perron-Frobenius theorem,  $\|\{S_n\}-\{T_+\}\|\lesssim  (d_p/\delta)^{-n}$ in the finite dimensional vector space $H^{p,p}(X,\R)$ (see also \cite[Section 3]{exp-kahler}). Therefore, we have $$\|(a_{n,1}-b_1,a_{n,2}-b_2,\dots,a_{n,t}-b_t)\|\lesssim  (d_p/\delta)^{-n}.$$ 
	
	Finally, observe that $\lp\alpha_j, \phi_\ep T_-\rp$ is uniformly bounded independent of $\phi$. Hence $$ |\lp S_n-T_+,  K_\ep\rp|\lesssim  (d_p/\delta)^{-n}.$$ The proof of this lemma is finished.
\end{proof}

\begin{lemma}\label{3.1}
	The function $g$ is H\"older continuous at $0$, more precisely, there exists a constant $c>0$ independent of $\phi$ such that for  $\ep\in (0,1/2]$, we have $|g(\ep)-g(0)|\leq c \ep^\alpha$ for some $0<\alpha\leq 1$.
\end{lemma}

\begin{proof}
  Since $T_+\wedge T_-$ has a H\"older continuous super-potential, by definition,  we have $$|g(\ep)-g(0)|\leq M\dist_2(\ddc \phi_{\ep},\ddc\phi)^\alpha$$ for some constants $0<\alpha\leq 1,M>0$.
  
    Since $\phi_\ep$ is decreasing when $\ep$ decreases, by definition and estimates \eqref{regul} we obtain 	\begingroup
    \allowdisplaybreaks
\begin{align*}
	\dist_2(\ddc\phi_\ep,\ddc \phi)&=\sup_{\|\Phi\|_{\Cc^2}\leq 1}|\lp \ddc\phi_\ep-\ddc\phi,\Phi\rp|=\sup_{\|\Phi\|_{\Cc^2}\leq 1}|\lp \phi_\ep-\phi,\ddc\Phi\rp| \\
	&\lesssim \lp \phi_\ep-\phi,\omega^k\rp= \|\phi_\ep-\phi\|_{L^1( \omega^k)}\lesssim \ep
\end{align*} \endgroup since $\|\Phi\|_{\Cc^2}\leq 1$ implies $\pm\ddc \Phi \leq c' \omega^k$, where $c'$ is a positive constant  only depending on   $(X,\omega)$.
 Therefore,
$$|g(\ep)-g(0)|\leq M\dist_2(\ddc \phi_{\ep},\ddc\phi)^\alpha\lesssim \ep^\alpha.$$ The proof of this lemma is complete.
\end{proof}

Since $\phi_\ep$ is smooth for every $\ep\neq 0$, in particular it is H\"older continuous. We can easily obtain the estimates of $h_n(\ep)-h(\ep)$ for $\ep\neq 0$ by using Proposition \ref{mainlemma}. Combining with the above two lemmas  we get the following key proposition.

\begin{proposition}\label{mainprop}
	Let $S_n$ and $\phi$ be as above. There exists a constant $c>0$ independent of $\phi$ such that 
$$\lp S_n\wedge T_-,\phi\rp-\lp T_+\wedge T_-,\phi\rp\leq c (d_p/\delta)^{-n}$$ for all $n$.
\end{proposition}

\begin{proof}
	Again, we fix a constant $\delta_0$ such that $\max\{d_{p-1},d_{p+1}\}<\delta_0<\delta$ and $\delta_0$ satisfies the same properties of $\delta$ as in Theorem \ref{main}.  By Lemma \ref{2.1}, $\|\ddc\phi_\ep\|_*\leq 2$ for all $\ep$, thus $\|\ddc\phi_\ep\wedge T_-\|_*$ are uniformly bounded for $1<\ep\leq 1/2$.
	Since  $\|\phi_\ep\|_{\Cc^2}\lesssim \ep^{-2}$ and $T_-$ has a H\"older continuous super-potential (see \cite{super-kahler}), by \cite[Proposition 3.4.2]{super-kahler}, $\ddc\phi_\ep\wedge T_-$ has a $(2,\lambda,M\ep^{-2})$-H\"older continuous super-potential for some constant $\lambda$ and $M$ independent of $\phi$. Multiplying $\phi$ by some constant allows us to assume $M=1$ and $\|\ddc\phi_\ep\wedge T_-\|_*\leq 1$ for all $0<\ep\leq 1/2$. Applying Proposition \ref{mainlemma} to the family $\{\ddc\phi_\ep\wedge T_-\}$ instead of $\{R_\ep\}$, we get that for $0<\ep\leq 1/2$, $$h_n(\ep)-h(\ep)\lesssim -\log\ep(d_p/\delta_0)^{-n},$$ where the $\lesssim$ is independent of $\phi$. Combining this with estimates \eqref{regul}, Lemma \ref{3.0} and Lemma \ref{3.1}, we have for  $\ep\in (0,1/2]$,
	\begingroup
	\allowdisplaybreaks
	\begin{align*}
&\lp S_n\wedge T_-,\phi\rp-\lp T_+\wedge T_-,\phi\rp\leq \lp S_n\wedge T_-,\phi_\ep\rp-\lp T_+\wedge T_-,\phi\rp\\
	&=\lp S_n\wedge T_-,\phi_\ep\rp-\lp T_+\wedge T_-,\phi_\ep\rp +\lp T_+\wedge T_-,\phi_\ep\rp-\lp T_+\wedge T_-,\phi\rp\\
	&=h_n(\ep)-h(\ep)+\lp S_n, K_\ep\rp-\lp T_+, K_\ep\rp+g(\ep)+\lp\omega^k,\phi_\ep\rp-g(0)-\lp\omega^k,\phi\rp\\
	&\lesssim -\log\ep(d_p/\delta_0)^{-n}+ (d_p/\delta_0)^{-n}+\ep^\alpha+\ep,
		\end{align*}\endgroup where the first inequality is because $\phi_\ep$ is decreasing as $\ep$ decreasing and $ S_n\wedge T_-$ is positive.
		
	Finally, since $\alpha\leq1$, by taking $\ep=(d_p/\delta_0)^{-n/\alpha}$, we get $$\lp S_n\wedge T_-,\phi\rp-\lp T_+\wedge T_-,\phi\rp\lesssim n\log(d_p/\delta_0)(d_p/\delta_0)^{-n}+(d_p/\delta_0)^{-n}\lesssim (d_p/\delta)^{-n}.$$ Since the constant $c$ in Lemma \ref{3.0} and Lemma \ref{3.1} are independent of $\phi$, the $\lesssim$ above is independent of $\phi$.
	\end{proof}

In Proposition \ref{mainprop}, the constant $c$ depends on $S$.  Note that we do not have a lower bound estimate. Otherwise, we can follow the approach in \cite{exp-kahler} to show Theorem \ref{main} directly. Here we need some extra techniques from \cite{wu-mix}.

\begin{proof}[Proof of Theorem \ref{main}]\phantom{\qedhere}
	 Multiplying $\varphi$ and $\psi$ by some constant allows us to assume $\|\varphi\|_{L^\infty}\leq 1/2$ and $\|\psi\|_{L^\infty}\leq 1/2$.	It is sufficient to prove Theorem \ref{main} for $n$ even because applying it to $\varphi$ and $\psi\circ f$ gives the case of odd $n$.  Using the invariance of  $\mu$, it is enough to show that 
	 \begin{equation}\label{4}
	 \big|\big\lp \mu,(\varphi\circ f^n)(\psi\circ f^{-n})\big\rp-\lp\mu,\varphi\rp\lp\mu,\psi\rp\big|\leq c(d_p/\delta)^{-n}
	 \end{equation} for some $c>0$. It is equivalent to prove \[\big\lp \mu,(\varphi\circ f^n)(\psi\circ f^{-n})\big\rp-\lp\mu,\varphi\rp\lp\mu,\psi\rp\leq c(d_p/\delta)^{-n}\] and \[\big\lp \mu,(\varphi\circ f^n)(-\psi\circ f^{-n})\big\rp-\lp\mu,\varphi\rp\lp\mu,-\psi\rp\leq c(d_p/\delta)^{-n}.\]
	 
	 \medskip
	 For  $j=1,2$, we define $$\varphi_j^+=\varphi^2+j\varphi+A,\quad \varphi_j^-=\varphi^2+j\varphi-A,\quad \psi_j^+=\psi^2+j\psi+A,\quad \psi_j^-=-\psi^2-j\psi+A,$$ where $A$ is a positive constant whose value will be determined later.
	 Consider the following eight functions on $X\times X$:
	 $$\Phi_{jl}^+(z,w)=\varphi_j^+(z)\psi_l^+(w),\quad \Phi_{jl}^-(z,w)=\varphi_j^-(z)\psi_l^-(w),$$ where $j,l=1,2$. We need the following lemma. \end{proof} 

\begin{lemma}\label{quasi-psh}
	The functions $\Phi_{jl}^\pm$ are quasi-p.s.h.\ on $X\times X$ for $A$ large enough.
\end{lemma}

\begin{proof}
	We only show $\Phi_{11}^+$ and $\Phi_{11}^-$ are quasi-p.s.h.\ because the other cases can be obtained in the same way. By a direct computation (see also   \cite[Lemma 3.1]{wu-mix}), we have
	\begingroup
	\allowdisplaybreaks
	\begin{align*}
		i\partial\dbar\Phi_{11}^+
	&=(\psi^2+\psi+A)\big((2\varphi+1)i\partial\dbar\varphi+2i\partial\varphi\wedge\dbar\varphi\big)+ (2\varphi+1)(2\psi+1)i\partial\varphi\wedge\dbar\psi  \\
		&\quad\,\, +(2\varphi+1)(2\psi+1)i\partial\psi\wedge\dbar\varphi+(\varphi^2+\varphi+A)\big((2\psi+1)i\partial\dbar\psi+2i\partial\psi\wedge\dbar\psi\big).
	\end{align*}\endgroup
Combining with the identity
\[i\partial\varphi\wedge\dbar\varphi+i\partial\varphi\wedge\dbar\psi+i\partial\psi\wedge\dbar\varphi+i\partial\psi\wedge\dbar\psi=i\partial(\varphi+\psi)\wedge\dbar(\varphi+\psi)\geq 0,\]  we get 	\begingroup
\allowdisplaybreaks
\begin{align*}
		i\partial\dbar\Phi_{11}^+&\geq (\psi^2+\psi+A)(2\varphi+1)i\partial\dbar\varphi+(\varphi^2+\varphi+A)(2\psi+1)i\partial\dbar\psi \\
		&\quad\,\, +\big(2\psi^2+2\psi +2A-(2\varphi+1)(2\psi+1)\big)i\partial\varphi\wedge\dbar\varphi\\
			&\quad\,\, +\big(2\varphi^2+2\varphi +2A-(2\varphi+1)(2\psi+1)\big)i\partial\psi\wedge\dbar\psi.
\end{align*}\endgroup
Recall that we assume $\|\varphi\|_{L^\infty}\leq 1/2$ and $\|\psi\|_{L^\infty}\leq 1/2$. So $2\varphi+1\geq 0, 2\psi +1\geq 0$.
	 We  take $A$ large enough such that $\psi^2+\psi+A,\varphi^2+\varphi+A,2\psi^2+2\psi +2A-(2\varphi+1)(2\psi+1),2\varphi^2+2\varphi +2A-(2\varphi+1)(2\psi+1)$ are all positive. Since $\varphi$ and $\psi$ are quasi-p.s.h.\ on X and $i\partial\varphi\wedge\dbar\varphi,i\partial\psi\wedge\dbar\psi$ are positive, we deduce that $\Phi_{11}^+$ is quasi-p.s.h.\ on $X\times X$.
	 
\medskip	 
For $\Phi_{11}^-$, we have
	\begingroup
\allowdisplaybreaks
\begin{align*}
i\partial\dbar\Phi_{11}^-
&=(-\psi^2-\psi+A)\big((2\varphi+1)i\partial\dbar\varphi+2i\partial\varphi\wedge\dbar\varphi\big)- (2\varphi+1)(2\psi+1)i\partial\varphi\wedge\dbar\psi  \\
&\quad\,\, -(2\varphi+1)(2\psi+1)i\partial\psi\wedge\dbar\varphi+(\varphi^2+\varphi-A)\big(-(2\psi+1)i\partial\dbar\psi-2i\partial\psi\wedge\dbar\psi\big).
\end{align*}\endgroup
Combining with the identity
\[i\partial\varphi\wedge\dbar\varphi-i\partial\varphi\wedge\dbar\psi-i\partial\psi\wedge\dbar\varphi+i\partial\psi\wedge\dbar\psi=i\partial(\varphi-\psi)\wedge\dbar(\varphi-\psi)\geq 0,\] we get
	\begingroup
\allowdisplaybreaks
\begin{align*}
i\partial\dbar\Phi_{11}^-&\geq (-\psi^2-\psi+A)(2\varphi+1)i\partial\dbar\varphi+(-\varphi^2-\varphi+A)(2\psi+1)i\partial\dbar\psi \\
&\quad\,\, +\big(-2\psi^2-2\psi +2A-(2\varphi+1)(2\psi+1)\big)i\partial\varphi\wedge\dbar\varphi\\
&\quad\,\, +\big(-2\varphi^2-2\varphi +2A-(2\varphi+1)(2\psi+1)\big)i\partial\psi\wedge\dbar\psi.
\end{align*}\endgroup Repeating the same argument as above, we get that $\Phi_{11}^-$ is quasi-p.s.h.\ for $A$ large enough. The proof is complete.
\end{proof}

We choose $A$ large enough such that all the $\Phi_{jl}^\pm$ are bounded and quasi-p.s.h.\ on $X\times X$. Note that the choice of $A$  is independent of $\varphi$ and $\psi$.  Define $\widetilde\omega:= \pi_1^*\omega+\pi_2^*\omega$, where $\pi_1,\pi_2$ are the two canonical projections of $X\times X$ onto its factors. Then $\widetilde\omega$ is the canonical K\"ahler form on $X\times X$. Recall that we assume  $\ddc\varphi\geq -\omega,\ddc\psi\geq -\omega$. From the computations in Lemma \ref{quasi-psh}, we deduce that $\ddc\Phi_{11}^+\geq -3A\widetilde\omega$ when $A$ is large. And observe that $\Phi_{11}^+$ is bounded by $4A^2$.\\

Next  we consider the automorphism $F$ of $X \times X$ which is defined by $$F(z,w):=\big(f^{-1}(z),f(w)\big).$$ By using K\"unneth formula, one can show that  the dynamic degree of order $k$ of $F$ is equal to $d_p^2$ (see also  \cite[Section 4]{exp-kahler}), and the dynamical degrees and the eigenvalues of $F^*$ on $H^{k,k}(X\times X,\R)$,  except $d_p^2$, are strictly smaller than $d_p\delta$. Hence $F$ and $d_p\delta$ satisfy the conditions of $f$ and $\delta$ respectively in Theorem \ref{main}.

It is not hard to see that the Green $(k,k)$-currents of $F$ and $F^{-1}$  are $T_-\otimes T_+$ and $T_+\otimes T_-$ respectively, and they satisfy $$F^*(T_-\otimes T_+)=d_p^2(T_-\otimes T_+),F_*(T_+\otimes T_-)=d_p^2(T_+\otimes T_-).$$ In particular, they have  H\"older continuous super-potentials.
Let $\Delta$ denote the diagonal of $X\times X$. Then $[\Delta]$ is a positive closed $(k,k)$-current on $X\times X$. With the help of $F$, we get the following estimates.

\begin{lemma}\label{3.4}
 There exists a constant $c>0$ such that $$\big\lp \mu,(\varphi_j^+\circ f^n)(\psi_l^+\circ f^{-n})\big\rp-\lp\mu,\varphi_j^+\rp\lp\mu,\psi_l^+\rp\leq c(d_p/\delta)^{-n}$$ and $$\big\lp \mu,(\varphi_j^-\circ f^n)(\psi_l^-\circ f^{-n})\big\rp-\lp\mu,\varphi_j^-\rp\lp\mu,\psi_l^-\rp\leq c(d_p/\delta)^{-n}$$ for all $j,l$ and $n$.
\end{lemma}

\begin{proof}
We only show this lemma holds for $\varphi_1^+$ and $\psi_1^+$, the proofs of others are similar. For the automorphism $F$, consider the sequence of currents $d_p^{-2n}(F^n)^*[\Delta]$, which are positive closed currents of mass $1$ converging to $T_-\otimes T_+$. Since $\ddc\Phi_{11}^+\geq -3A\widetilde\omega$ and $|\Phi_{11}^+|\leq 4A^2$, after dividing $\Phi_{11}^+$ by $4A^2$,  we can assume $\ddc\Phi_{11}^+\geq -\widetilde\omega$ and $|\Phi_{11}^+|\leq 1$. Applying Proposition \ref{mainprop} to $d_p^{-2n}(F^n)^*[\Delta],T_+\otimes T_-$ and $\Phi_{11}^+$ instead of $S_n,T_-$ and $\phi$, we deduce that there exists a constant $c>0$ such that $$\big\lp d_p^{-2n}(F^n)^*[\Delta]\wedge (T_+\otimes T_-),\Phi_{11}^+\big\rp-\big\lp (T_-\otimes T_+)\wedge(T_+\otimes T_-),\Phi_{11}^+\big\rp\leq c \big(d_p^2/(d_p\delta)\big)^{-n}$$ for  all $n$. Here $c$ is independent of $\varphi$ and $\psi$ because $A$ is independent of them.

On the other hand, by definition, we have 	\begingroup
\allowdisplaybreaks
\begin{align*}
\big\lp d_p^{-2n}(F^n)^*[\Delta]\wedge (T_+\otimes T_-),\Phi_{11}^+\big\rp&=\big\lp [\Delta],d_p^{-2n}(F^n)_*\big[(T_+\otimes T_-)\wedge \Phi_{11}^+\big]\big\rp \\
&=\big\lp[\Delta]\wedge (T_+\otimes T_-),\Phi_{11}^+\circ F^{-n}\big\rp\\
&=\big\lp T_+\wedge T_-,(\varphi_1^+\circ f^n)(\psi_1^+\circ f^{-n})\big\rp,
\end{align*}\endgroup
and $$\big\lp (T_-\otimes T_+)\wedge(T_+\otimes T_-),\Phi_{11}^+\big\rp=\lp \mu \otimes \mu,\Phi_{11}^+\rp=\lp\mu,\varphi_1^+\rp\lp\mu,\psi_1^+\rp.$$
This finishes the proof of this lemma.
\end{proof}

	Now we can finish the proof of Theorem \ref{main} using the invariant property of $\mu$.

\begin{proof}[End of the proof of Theorem \ref{main}]
 Consider $\alpha_{11}^+=2, \alpha_{22}^+=\alpha_{11}^-=\alpha_{21}^-=\alpha_{12}^-=1$ and $\alpha_{21}^+=\alpha_{12}^+=\alpha_{22}^-=0$. A direct computation gives 	\begingroup\allowdisplaybreaks \begin{align*}
	\sum_{j,l=1,2} &\Big(\alpha_{jl}^+(\varphi_j^+\circ f^n)(\psi_l^+\circ f^{-n})+\alpha_{jl}^-(\varphi_j^-\circ f^n)(\psi_l^-\circ f^{-n})\Big) \\
	&=(\varphi\circ f^n)(\psi\circ f^{-n})+\beta_1\varphi^2\circ f^n+\beta_2\psi^2\circ f^{-n}+\beta_3\varphi\circ f^n+\beta_4\psi\circ f^{-n}+\beta_5 
	\end{align*}\endgroup
	for some constants $\beta_t,1\leq t\leq 5$.  We now apply this identity and Lemma \ref{3.4}. Observe that the invariance of $\mu$ implies that
	$$\lp \mu,\varphi^m\circ f^{\pm n}\rp=\lp\mu,\varphi^m\rp\quad\text{and}\quad \lp \mu,\psi^m\circ f^{\pm n}\rp=\lp\mu,\psi^m\rp.$$
	Hence the terms involving $\beta_t$ cancel each other out. We obtain
	\[\big\lp \mu,(\varphi\circ f^n)(\psi\circ f^{-n})\big\rp-\lp\mu,\varphi\rp\lp\mu,\psi\rp\leq\Big(\sum_{j,l=1,2}\big(\alpha_{jl}^++\alpha_{jl}^-\big)\Big)c(d_p/\delta)^{-n}=6c(d_p/\delta)^{-n}.\]
	
	Similarly, taking $\gamma_{11}^-=2, \gamma_{11}^+=\gamma_{21}^+=\gamma_{12}^+=\gamma_{22}^-=1$ and $\gamma_{22}^+=\gamma_{21}^-=\gamma_{12}^-=0$, we get  \[\big\lp \mu,(\varphi\circ f^n)(-\psi\circ f^{-n})\big\rp-\lp\mu,\varphi\rp\lp\mu,-\psi\rp\leq \Big(\sum_{j,l=1,2}\big(\gamma_{jl}^++\gamma_{jl}^-\big)\Big)c(d_p/\delta)^{-n}=6c(d_p/\delta)^{-n}.\]
	The above two inequalities imply  inequality \eqref{4} and finish the proof of Theorem \ref{main}. 
\end{proof}

Using the moderate property of $\mu$ and the technical of replacing $\delta$ by $\delta_0$, we can prove Theorem \ref{second-theorem}. 

\begin{proof}[Proof of Theorem \ref{second-theorem}]
	It is enough to prove this theorem for all negative quasi-p.s.h. functions $\varphi$ and $\psi$. Multiplying them by some constant allows us to assume $\ddc\varphi\geq -\omega,\ddc\psi\geq -\omega$ and $\lp\mu,|\varphi|\rp\leq 1,\lp\mu,|\psi|\rp\leq 1$. Define \[\varphi_1:=\max\{\varphi,-M\}, \quad \psi_1:=\max\{\psi,-M\},\]
	and \[\varphi_2:=\varphi-\varphi_1,\quad \psi_2:=\psi-\psi_1.\]
	
		Then $\varphi_1$ and $\psi_1$ are bounded  quasi-p.s.h.\ functions which satisfy  $\ddc\varphi_1\geq -\omega,\ddc\psi_1\geq -\omega$. Fix a constant $\delta_0$ such that $\max\{d_{p-1},d_{p+1}\}<\delta_0<\delta$ and $\delta_0$ satisfies the same properties of $\delta$ as in Theorem \ref{main}. Applying Theorem  \ref{main} to $\varphi_1$ and $\psi_1$,  we get  \[\Big|\int(\varphi_1\circ f^n)\psi_1 d\mu -\Big(\int\varphi_1 d\mu\Big)\Big(\int \psi_1 d\mu\Big)\Big|\lesssim (d_p/\delta_0)^{-n/2}M^2.\] 
		
		On the other hand, since $\mu$ is moderate, by  \cite[Lemma 2.1]{monge-measure} or the proof of  \cite[Theorem 1.3]{wu-mix}, we  get for some $\alpha>0$, \[\|\varphi_2\|_{L^1(\mu)}\lesssim e^{-\alpha M/2},	\|\psi_2\|_{L^1(\mu)}\lesssim e^{-\alpha M/2},\|\varphi_2\|_{L^2(\mu)}\lesssim e^{-\alpha M/2},	\|\psi_2\|_{L^2(\mu)}\lesssim e^{-\alpha M/2}.\]

	 From the invariance of $\mu$, we have that $\|\varphi_2\circ f^n\|_{L^p(\mu)}=\|\varphi_2\|_{L^p(\mu)}$ and $\|\psi_2\circ f^n\|_{L^p(\mu)}=\|\psi_2\|_{L^p(\mu)}$ for $1\leq p\leq \infty$. We do the following direct computation (see also \cite[Theorem 1.3]{wu-mix}),
	 \begingroup
	 \allowdisplaybreaks
	 \begin{align*}
	 &\big|\big\lp\mu,(\varphi\circ f^n)\psi\big\rp -\lp\mu,\varphi\rp\lp\mu, \psi \rp\big| \\
	 &=\big|\big\lp\mu, (\varphi_1\circ f^n+\varphi_2\circ f^n) (\psi_1+\psi_2)\big\rp -\lp\mu, \varphi_1+\varphi_2\rp \lp\mu,\psi_1+\psi_2\rp\big|\\
	 &\leq \big|\big\lp\mu,(\varphi_1\circ f^n)\psi_1\big\rp  -\lp\mu,\varphi_1\rp\lp\mu, \psi_1 \rp\big| +\big|\big\lp \mu, (\varphi_1\circ f^n)\psi_2\big\rp\big| +\big|\big\lp\mu, (\varphi_2\circ f^n)\psi_1\big\rp\big|\\
	 &\quad\,\,+\big|\big\lp\mu, (\varphi_2\circ f^n)\psi_2\big\rp\big|+|\lp\mu,\varphi_2\rp\lp\mu,\psi_1\rp|+|\lp\mu,\varphi_1\rp\lp\mu,\psi_2\rp|+|\lp\mu,\varphi_2\rp\lp\mu,\psi_2\rp| \\
	 & \leq\big|\big\lp\mu,(\varphi_1\circ f^n)\psi_1\big\rp  -\lp\mu,\varphi_1\rp\lp\mu, \psi_1 \rp\big|+M\|\varphi_2\|_{L^1(\mu)}+M\|\psi_2\|_{L^1(\mu)} \\
	 & \quad\,\,+\|\varphi_2\|_{L^2(\mu)}\|\psi_2\|_{L^2(\mu)}+\|\varphi_2\|_{L^1(\mu)} +\|\psi_2\|_{L^1(\mu)}+\|\varphi_2\|_{L^1(\mu)}\|\psi_2\|_{L^1(\mu)}              \\
	 &\lesssim  (d_p/\delta_0)^{-n/2}M^2+(2M+2 )e^{-\alpha M/2}+2e^{-\alpha M}. 
	 \end{align*}
	 \endgroup 
	 
	 Taking $M=\big(n\log (d_p/\delta_0)\big)/\alpha$, we obtain the estimate \[(d_p/\delta_0)^{-n/2}M^2+(2M+2)e^{-\alpha M/2}+2e^{-\alpha M}\lesssim n^2(d_p/\delta_0)^{-n/2}\lesssim (d_p/\delta)^{-n/2}.\] Therefore, \[\Big|\int(\varphi\circ f^n)\psi d\mu -\Big(\int\varphi d\mu\Big)\Big(\int \psi d\mu\Big)\Big|\lesssim (d_p/\delta)^{-n/2}.\] The proof is finished.
\end{proof}

\begin{remark}\rm
In the last step of the proof above, there is an $n^2$ appearing in the middle before replacing $\delta_0$ by $\delta$. It somehow represents the singularities of $\varphi$ and $\psi$. The constant $c$ in Theorem \ref{main} and Theorem \ref{second-theorem} can be made more explicit, but it needs a long calculation so we chose not to do here.
\end{remark}

\vspace{0.5cm}


\end{document}